\documentclass[12pt,a4paper]{article}
\usepackage{amssymb,amsmath,amsthm}
\usepackage{graphicx}
\usepackage{tikz}
\usetikzlibrary{patterns}
\usetikzlibrary{arrows}
\usetikzlibrary{decorations.markings}
\usetikzlibrary{arrows}
\usepackage{hyperref}
\usepackage{xcolor}

\newcommand\ex{\ensuremath{\mathrm{ex}}}
\newcommand\exo{\ensuremath{\mathrm{ex_o}}}
\newcommand\exd{\ensuremath{\mathrm{ex_d}}}

\theoremstyle{plain}
\newtheorem{theorem}{Theorem}[section]

\newtheorem{proposition}[theorem]{Proposition}

\theoremstyle{definition}

\newtheorem{definition}[theorem]{Definition}

\newtheorem{claim}[theorem]{Claim}

\newtheorem*{lemma*}{Lemma}
\newtheorem*{thm*}{Theorem}

\newcommand\cref[1]{Corollary~\ref{cor:#1}}

\title{On the supersaturation of oriented Tur\'an problems}

\author{Xuanrui Hu$^{1}$, Yuefang Sun$^{2}$
\\
$^{1}$ School of Mathematics and Statistics, Ningbo University,\\
Ningbo 315211, China, huxuanrui1108@163.com\\
$^{2}$ Corresponding author. School of Mathematics and Statistics,\\
Ningbo University,
Ningbo 315211, China, sunyuefang@nbu.edu.cn}
\date{}

\begin{document}

\maketitle

\begin{abstract}
The oriented Tur\'{a}n number of a given oriented graph $\overrightarrow{F}$, denoted by $\exo(n,\overrightarrow{F})$, is the largest number of arcs in $n$-vertex $\overrightarrow{F}$-free oriented graphs. This parameter could be seen as a natural oriented version of the classical Tur\'{a}n number.
In this paper, we study the supersaturation phenomenon for oriented Tur\'{a}n problems, and prove oriented versions of the famous Erd\H{o}s-Simonovits Supersaturation Theorem and Moon-Moser inequality, and supersaturation theorems for tournaments and antidirected complete bipartite graphs.
\end{abstract}
\vspace{0.3cm}


\vspace{0.3cm}

\section{Introduction}


In this paper, we only consider digraphs without loops or parallel arcs. For a set $S$, we use $|S|$ to denote its size. We use $[m]$ to denote the set of positive integers from $1$ to $m$ and use $T(n,k)$ to denote the {\em Tur\'{a}n graph} which is the complete $k$-partite graph on $n$ vertices whose partition sets differ in size by at most one. 

In a digraph $G$, the {\em total degree}, denoted by $d(v)$, of a vertex $v$ is the sum of the out-degree and in-degree of $v$ in $G$.
A digraph is {\em acyclic} if it contains no directed cycles. An {\em oriented graph} is a digraph without directed cycles of length two. An oriented graph is \textit{antidirected} if it contains only sources and sinks, that is, it is bipartite such that all arcs have the same direction. 
We use $\overrightarrow{H}=\overrightarrow{H}(A\cup B,A\rightarrow B)$ to denote an antidirected oriented graph $\overrightarrow{H}$ with bipartition $A\cup B$ such that all arcs of $\overrightarrow{H}$ are from $A$ to $B$. In particular, when $|A|=s, |B|=t$ and $H$ is complete, $\overrightarrow{H}(A\cup B,A\rightarrow B)$ is abbreviated as $\overrightarrow{K}_{s,t}$. 
We will use the shorthand name {\em $n$-tournament} for a tournament on $n$ vertices. An $n$-tournament $T$ is {\em transitive} if for any three distinct vertices $u, v, w$ in $T$, if there is an arc from $u$ to $v$ and from $v$ to $w$, then there must also be an arc from $u$ to $w$. Observe that the transitive tournament with order $n$ has an ordering $u_1, u_2, \ldots, u_n$ of its vertices such that $\overrightarrow{u_iu_j}$ is an arc whenever $1\leq i< j\leq n$. Hence, it is a unique acyclic tournament with order $n$, and we denote it by $\overrightarrow{TT}_n$.  


A {\em homomorphism} from a digraph $H$ to another digraph $D$ is a mapping $f:V(H)\rightarrow V(D)$ such that $\overrightarrow{f(u)f(v)}\in E(D)$ whenever $\overrightarrow{uv}\in E(H)$. A digraph $D$ is {\em $H$-free} if there is no copy of $H$ in $D$. 
The \textit{compressibility} $z(\overrightarrow{F})$ of an oriented graph $\overrightarrow{F}$ is the smallest $k$ such that there is a homomorphism from $\overrightarrow{F}$ to any tournament of order $k$. 

\begin{definition}[Oriented Tur\'{a}n number]\label{def3}
The {\em oriented Tur\'{a}n number} of a given oriented graph $\overrightarrow{F}$ is the largest number of arcs in $n$-vertex $\overrightarrow{F}$-free oriented graphs. This quantity is denoted by $\exo(n,\overrightarrow{F})$.
\end{definition}

The concept of oriented Tur\'{a}n number $\exo(n,\overrightarrow{F})$ could be seen as an oriented version of the classical Tur\'{a}n number (of $F$), denoted by $\ex(n,F)$, which is the largest number of edges in $n$-vertex graphs that do not contain $F$ as a subgraph.
There is another related quantity $\exd(n,H)$, which is the largest number of arcs in $n$-vertex \emph{digraphs} that do not contain a \emph{digraph} $H$ as a subdigraph. Clearly, $$\exo(n,\overrightarrow{F})\le \exd(n,\overrightarrow{F})\le 2\exo(n,\overrightarrow{F})$$ for an oriented graph $\overrightarrow{F}$. Let us remark that $\exd(n,H)$ has been more widely studied, first by Brown and Harary \cite{bh}.

It is natural to extend the concepts of Tur\'{a}n density and edge density of graphs to oriented graphs as follows. Note that the oriented Tur\'{a}n density is well-defined by Proposition~\ref{pi}.

\begin{definition}[Oriented Tur\'{a}n density]\label{def1}
The {\em oriented Tur\'{a}n density} of a given oriented graph $\overrightarrow{F}$ is $$\pi_o(\overrightarrow{F})=\lim_{n\to \infty} \frac{\exo(n,\overrightarrow{F})}{\binom{n}{2}}.$$
\end{definition}

\begin{definition}[Oriented edge density]\label{def2}
The {\em oriented edge density} of a given $n$-vertex oriented graph $\overrightarrow{F}$ is $$\tau_o(\overrightarrow{F})=\frac{|E(\overrightarrow{F})|}{\binom{n}{2}}.$$
\end{definition}

The celebrated Erdős-Stone-Simonovits Theorem~\cite{Erdos-Stone, Erdos-Simonovits} states that for any graph $F$, we have $\ex(n,F)=\left(\frac{\chi(F)-2}{\chi(F)-1}+o(1)\right)\binom{n}{2}$, which determines the asymptotics of $\ex(n,F)$ when $\chi(F)\geq 3$.

Valadkhan~\cite{vala} initially studied $\exo(n,\overrightarrow{F})$ and proved an analogue of the Erdős-Stone-Simonovits Theorem~\cite{Erdos-Stone, Erdos-Simonovits} in terms of the compressibility $z(\overrightarrow{F})$ which plays a similar role as the chromatic number $\chi(F)$ in Erdős-Stone-Simonovits Theorem. 

\begin{theorem}[Valadkhan~\cite{vala}]\label{OESS} For any acyclic oriented graph $\overrightarrow{F}$, we have 
\[\exo(n,\overrightarrow{F})=\left(\frac{z(\overrightarrow{F})-2}{z(\overrightarrow{F})-1}+o(1)\right)\binom{n}{2}.\]
\end{theorem}

The above theorem determines the asymptotics of $\exo(n,\overrightarrow{F})$ when $z(\overrightarrow{F})\geq 3$. If $\overrightarrow{F}$ contains a directed cycle, then the transitive tournament is $\overrightarrow{F}$-free, thus $\exo(n, \overrightarrow{F})=\binom{n}{2}$. It was shown that $2^{\frac{r-1}{2}}\leq z(\overrightarrow{TT}_r)\leq 2^{r-1}$ (see e.g. Chapter 3.2 of~\cite{vala}).
If $z(\overrightarrow{F})=2$, then the above theorem implies only $\exo(n,\overrightarrow{F})=o(n^2)$. Note that $z(\overrightarrow{F})=2$ if and only if $\overrightarrow{F}$ is antidirected, and this situation corresponds to the degenerate case of classical Turán number $\ex(n,F)$ when $\chi(F)=2$. Therefore, it is natural to try to obtain nice bounds for $\exo(n,\overrightarrow{F})$ when $\overrightarrow{F}$ is antidirected, or exact values for $\exo(n,\overrightarrow{F})$ of some oriented graphs $\overrightarrow{F}$. 

There are some existing results for upper bounds of $\exd(n,\overrightarrow{F})$ when $\overrightarrow{F}$ is antidirected, and these bounds clearly hold for $\exo(n,\overrightarrow{F})$ by the fact that $\exo(n,\overrightarrow{F})\le \exd(n,\overrightarrow{F})$. Graham \cite{gra} showed that $\exd(n,\overrightarrow{F})=O(n)$ for each antidirected tree $\overrightarrow{F}$ with $k$ arcs where $k\geq 2$. The constant factor was improved later, to $\exd(n,\overrightarrow{F})\le 4kn$ by Burr \cite{burr}. Addario-Berry, Havet, Linhares Sales, Reed and Thomass\'e \cite{ahlrt} conjectured that $\exd(n,\overrightarrow{F})\le (k-1)n$ and proved this for trees of diameter at most 3. Klimo\u{s}ov\'{a} and Stein \cite{ks} showed that $\exo(n,\overrightarrow{F})\leq (3k-4)n/2$ for the antidirected path $\overrightarrow{F}$ with $k$ arcs, where $k\geq 3$. Recently, Grzesik and Skrzypczyk \cite{gs} showed that $\exo(n,\overrightarrow{F})\le (k-1+\sqrt{k-3})n$ for the antidirected path $\overrightarrow{F}$ with $k$ arcs, where $k\geq 4$.

The following result by F\"{u}redi~\cite{Furedi}, and Alon, Krivelevich and Sudakov~\cite{Alon-Krivelevich-Sudakov} generalizes the famous K{\H{o}}v{\'a}ri-S\'os-Tur\'an Theorem \cite{kst}:

\begin{theorem}[F\"{u}redi-Alon-Krivelevich-Sudakov~\cite{Furedi, Alon-Krivelevich-Sudakov}]\label{FAKS}
Let $H$ be a bipartite graph with bipartition $A\cup B$ where each vertex in $A$ has degree at most $r$ in $B$. Then there exists a constant $c=c(H)$ such that $$\ex(n,H)\le cn^{2-\frac{1}{r}}.$$    
\end{theorem}

Gerbner, Hu and Sun~\cite{Gerbner-Hu-Sun} studied the oriented Tur\'{a}n number of antidirected oriented graph and obtained the following oriented strengthening of the F\"{u}redi-Alon-Krivelevich-Sudakov Theorem.

\begin{theorem}[Gerbner-Hu-Sun~\cite{Gerbner-Hu-Sun}]\label{orientedFAKS}
Let $\overrightarrow{H}=\overrightarrow{H}(A\cup B,A\rightarrow B)$ be a bipartite oriented graph where each vertex in $A$ has out-degree at most $r$ in $B$. Then there exists a constant $c=c(\overrightarrow{H})$ such that $$\exo(n,\overrightarrow{H})\le cn^{2-\frac{1}{r}}.$$
\end{theorem}

In the same paper, Gerbner, Hu and Sun~\cite{Gerbner-Hu-Sun} also proved several propositions that give exact results for several oriented graphs. In particular, they determined all exact values of $\exo(n,\overrightarrow{F})$ for every oriented graph $\overrightarrow{F}$ with at most three arcs and sufficiently large $n$. In addition, they proved a stability result and used it to determine the Turán number of an orientation of $C_4$.

\subsection{Our results}
In this paper, we study the supersaturation phenomenon for oriented Tur\'{a}n problems. The following result concerns the following supersaturation phenomenon: if an oriented graph $\overrightarrow{G}$ has oriented edge density $\tau_o(\overrightarrow{G})$ above $\pi_o(\overrightarrow{F})$, then it has ``many" copies of $\overrightarrow{F}$, that is, there is a positive density number of $\overrightarrow{F}$ in $\overrightarrow{G}$. This result could be seen as an oriented version of the famous Erd\H{o}s-Simonovits Supersaturation Theorem~\cite{Erdos-Simonovits1984}.

\begin{theorem}\label{superori}
Let $\varepsilon >0$. For any oriented graph $\overrightarrow{F}$ with order $h$, there exist $\delta>0$ and $n_0\in \mathbb{N}$ such that any oriented graph $\overrightarrow{G}$ with order at least $n_0$ and size at least $
\left(\pi_o(\overrightarrow{F})+\varepsilon\right)\binom{n}{2}$ contains at least $\delta \binom{n}{h}$ copies of $\overrightarrow{F}$.
\end{theorem}

After that, we consider the transitive tournament $\overrightarrow{TT}_r$. When $r=3$, by Proposition~\ref{smallgraph}, we have $z(\overrightarrow{TT}_3)=4$. Gerbner, Hu and Sun~\cite{Gerbner-Hu-Sun} proved that $\exo(n,\overrightarrow{F})=|E(T(n,z(\overrightarrow{F})-1))|$ for any tournament $\overrightarrow{F}$. Therefore, $\exo(n,\overrightarrow{TT}_3)=|E(T(n,3))|$. The following result shows that an oriented graph with order $n$ and size $|E(T(n,3))|+1$ has at least about ``$2n/3$" copies of $\overrightarrow{TT}_3$.  Therefore, it could be seen as an oriented version of Rademacher's Theorem (see e.g. \cite{Erdos1962}) which shows that there are at least $\lfloor \frac{n}{2} \rfloor$ copies of triangles in an $n$-vertex graph $G$ with size $|E(T(n,2))|+1$. Note that there are two types of orientations of a triangle, another one is $\overrightarrow{C}_3$. However, it makes no sense to study the supersaturation for $\overrightarrow{C}_3$, as $\exo(n,\overrightarrow{C}_3)={n\choose 2}$.
\begin{theorem}\label{TT3}
Let $\overrightarrow{G}$ be an oriented graph with order $n=3k+t>3$, where $k$ is a positive integer and $t\in\{0,1,2\}$. If $|E(\overrightarrow{G})|=|E(T(n,3))|+1$, then there are at least $f(k)$
copies of $\overrightarrow{TT}_3$ in $\overrightarrow{G}$, where $f(k)=\begin{cases}
       2k,& t\in \{0,1\}\\
       2k+1,&t=2\\
    \end{cases}$.    
\end{theorem}

For general $r$, we prove an oriented version of Moon-Moser inequality~\cite{Moon-Moser} as follows:

\begin{theorem}\label{OMM}
For any oriented graph $\overrightarrow{G}$ with order $n$, denote by $N_r$ and $M_r$ the numbers of $r$-tournaments and copies of $\overrightarrow{TT}_r$ contained in $\overrightarrow{G}$, respectively. The following inequalities hold:
\begin{description}
\item[(a)]$\frac{N_{r+1}}{N_{r}}\geq \frac{1}{r^2-1}\left(r^2\frac{N_{r}}{N_{r-1}}-n\right)$ for $r\geq 2$;
\item[(b)]$\frac{M_{r+1}}{M_{r}}\geq \frac{1}{r^2-1}\left(r^2\frac{M_{r}}{M_{r-1}}-n\right)$ for $r\geq 3$.
\end{description}
\end{theorem}

By Theorem~\ref{OMM}, we give the following supersaturation result of $r$-tournaments for general $r$ in terms of oriented edge density $\tau_o(\overrightarrow{F})$: 

\begin{theorem}\label{TTdensity}
Let $\overrightarrow{G}$ be an oriented graph with order $n$.
If $$|E(\overrightarrow{G})|\geq(1-\frac{1}{t})\frac{n^2}{2}$$ for some $t\in \mathbb{R}^+$, then $$N_r(\overrightarrow{G})\geq \binom{t}{r}(\frac{n}{t})^r,$$ where $$\binom{t}{r}= \left\{\begin{matrix}
\frac{t(t-1)\cdots (t-r+1)}{r!}, & t>r-1, \\
0, & t\leq r-1.
\end{matrix}\right.$$
\end{theorem}

Finally, we turn our attention to the antidirected complete bipartite graphs $\overrightarrow{K}_{s,t}$.
By Theorem~\ref{orientedFAKS}, for positive integers $s,t\geq 1$, we have $\exo(n,\overrightarrow{K}_{s,t})\le cn^{2-\frac{1}{t}}$ for some constant $c$. We obtain the following supersaturation result for $\overrightarrow{K}_{s,t}$:

\begin{theorem}\label{antidirected4-cycle}
Let $s, t\geq 1$ be two positive integers. For every $n$-vertex oriented graph $\overrightarrow{G}$ with sufficiently large $n$, if $|E(\overrightarrow{G})|\geq es^{\frac{1}{t}}n^{2-\frac{1}{t}}$, then there are at least $\big(\frac{e}{t}\big)^t n^t$ copies of $\overrightarrow{K}_{s,t}$ in $\overrightarrow{G}$, where $e$ denotes Euler's number.
\end{theorem}


\section{Proof of Theorem~\ref{superori}}

The following result, which could be seen as an oriented version of Katona-Nemetz-Simonovits Theorem \cite{Katona-Nemetz-Simonovits} on classical Tur\'{a}n density, means that the oriented Tur\'{a}n density $\pi_o(\overrightarrow{F})$ is well defined. 

\begin{proposition}\label{pi}
For any oriented graph $\overrightarrow{F}$, $\pi_o(\overrightarrow{F})$ exists.
\end{proposition}
\begin{proof}
Let $a_n=\frac{\exo(n,\overrightarrow{F})}{\binom{n}{2}}$. Clearly, $a_n\in [0,1]$ for each $n\in \mathbb{N}$. It suffices to show that this bounded positive sequence $\{a_n\}_{n\in \mathbb{N}}$ is non-increasing. Let $\overrightarrow{G}_n$ be an $n$-vertex $\overrightarrow{F}$-extremal oriented graph, that is, $|E(\overrightarrow{G}_n)|=\exo(n,\overrightarrow{F})$. Hence, $a_n=\frac{|E(\overrightarrow{G}_n)|}{\binom{n}{2}}$.

Let $\overrightarrow{G}_{n-1}$ be a (random) oriented graph obtained by deleting a vertex from $\overrightarrow{G}_n$ uniformly at random. For each $e\in E(\overrightarrow{G}_n)$, the probability that the deleted vertex is an endvertex of $e$ is $\frac{2}{n}$, and so 
$$\mathbb{P}(e\in E(\overrightarrow{G}_{n-1}))=1-\frac{2}{n}.$$
 By linearity of expectation, we have
$$\mathbb{E}(|E(\overrightarrow{G}_{n-1})|)=(1-\frac{2}{n})|E(\overrightarrow{G}_{n})|=\frac{n-2}{n}a_n\binom{n}{2}=\binom{n-1}{2}a_n.$$
Therefore, there exists a choice of $v$ such that $\overrightarrow{G}_{n-1}=\overrightarrow{G}_n-v$ has at least $\binom{n-1}{2}a_n$ arcs. As $\overrightarrow{G}_{n-1}$ is still $\overrightarrow{H}$-free, we have
$$a_n=\frac{\binom{n-1}{2}a_n}{\binom{n-1}{2}} \leq \frac{|E(\overrightarrow{G}_{n-1})|}{\binom{n-1}{2}}\leq \frac{\exo(n-1,\overrightarrow{F})}{\binom{n-1}{2}}=a_{n-1}.$$
This completes the proof.
\end{proof}

We are now in a position to prove Theorem~\ref{superori}. Recall the theorem.

\vskip 0.5cm

\noindent
{\bf Theorem~\ref{superori}:} {\em   Let $\varepsilon >0$. For any oriented graph $\overrightarrow{F}$ with order $h$, there exist $\delta>0$ and $n_0\in \mathbb{N}$ such that any oriented graph $\overrightarrow{G}$ with order at least $n_0$ and size at least $
\left(\pi_o(\overrightarrow{F})+\varepsilon\right)\binom{n}{2}$ contains at least $\delta \binom{n}{h}$ copies of $\overrightarrow{F}$.}

\vskip 0.5cm

\begin{proof}
By the definition of $\pi_o(\overrightarrow{F})$, we can take a sufficiently large $m$ (depending only on $\varepsilon$ and $\overrightarrow{F}$), such that $$\exo(m,\overrightarrow{F})\leq \left(\pi_o(\overrightarrow{F})+\varepsilon/4\right)\binom{m}{2}.$$

We need the following claim:

\begin{claim}
There are at least $\frac{\varepsilon}{4}\binom{n}{m}$ $m$-sets in $V(\overrightarrow{G})$ inducing an oriented subgraph with more than $\left(\pi_o(\overrightarrow{F})+\varepsilon/4\right)\binom{m}{2}$ arcs.
\end{claim}

\begin{proof}
Suppose for contradiction. The following holds:
\begin{equation*}
\begin{aligned}
\sum_{M\in\binom{V(\overrightarrow{G})}{m}}|E(\overrightarrow{G}[M])|
 &\leq \frac{\varepsilon}{4} \binom{n}{m}\binom{m}{2}+\binom{n}{m} \left(\pi_o(\overrightarrow{F})+\frac{\varepsilon}{4}\right)\binom{m}{2}\\
 &=\binom{m}{2}\binom{n}{m}\left(\pi_o(\overrightarrow{F})+\frac{\varepsilon}{2}\right)\\
 &=\binom{n}{2}\binom{n-2}{m-2}\left(\pi_o(\overrightarrow{F})+\frac{\varepsilon}{2}\right),\\
\end{aligned}
\end{equation*}
On the other hand, we have
$$\sum_{M\in\binom{V(\overrightarrow{G})}{m}}|E(\overrightarrow{G}[M])|=\binom{n-2}{m-2}|E(\overrightarrow{G})|\geq \left(\pi_o(\overrightarrow{F})+\varepsilon\right)\binom{n}{2}\binom{n-2}{m-2}.$$
Hence we can get $\pi_o(\overrightarrow{F})+\frac{\varepsilon}{2}\geq \pi_o(\overrightarrow{F})+\varepsilon,$ a contradiction.
\end{proof}

Now on the one hand, by the above claim, there are at least $\frac{\varepsilon}{4}\binom{n}{m}$ $m$-sets in $V(\overrightarrow{G})$ inducing an oriented subgraph with more than $\left(\pi_o(\overrightarrow{F})+\varepsilon/4\right)\binom{m}{2}$ arcs, moreover, each such an oriented subgraph contains a copy of $\overrightarrow{F}$ by the definition of $\pi_o(\overrightarrow{F})$. On the other hand, each copy of $\overrightarrow{F}$ is contained in at most $\binom{n-h}{m-h}$ oriented subgraphs induced by $m$-sets. Hence, the number of copies of $\overrightarrow{F}$ in $\overrightarrow{G}$ is at least 
\begin{equation*}
\begin{aligned}
\frac{\frac{\varepsilon}{4}\binom{n}{m}}{\binom{n-h}{m-h}}&=\frac{\frac{\varepsilon}{4}\frac{n!}{m!(n-m)!}}{\frac{(n-h)!}{(n-m)!(m-h)!}}\\
&=\frac{\frac{\varepsilon}{4}n!}{m!}\frac{(m-h)!}{(n-h)!}\\
&=\frac{\frac{\varepsilon}{4}(m-h)!h!}{m!}\frac{n!}{(n-h)!h!}\\
&=\frac{\varepsilon}{4\binom{m}{h}}\binom{n}{h}.
\end{aligned}
\end{equation*} 
Now we finish the proof by setting $\delta=\frac{\varepsilon}{4\binom{m}{h}}.$
\end{proof}

\section{Proof of Theorem~\ref{TT3}}

We need to prove the following result at first.
\begin{proposition}\label{smallgraph}
The following assertions hold:
\begin{description}
\item[(a)] $z(\overrightarrow{TT}_3)=4$.
\item[(b)] Let $\overrightarrow{G}$ be an oriented graph with order $n\in \{4,5,6\}$. If $|E(\overrightarrow{G})|=|E(T(n,3))|+1$, then there are at least  $n-2$ copies of $\overrightarrow{TT}_3$ in $\overrightarrow{G}$.
\end{description}
\end{proposition}
\begin{proof}

\noindent{\bf Part (a)} 
Observe that any tournament with order four contains at least two copies of $\overrightarrow{TT}_3$. Indeed, let $T$ be a tournament on $\{v_i\mid i\in [4]\}$. Since $\sum_{i=1}^4d_T^+(v_i)=6$, one of the following two cases must hold: (i) there is a vertex, say $v_1$, with out-degree three; (ii) there are two vertices, say $v_2,v_3$, with out-degrees two. For the former case, $v_1$ and two of its out-neighbors form a $\overrightarrow{TT}_3$, and so there are at least three copies of $\overrightarrow{TT}_3$ in $T$; for the latter case, $v_i~(2\leq i\leq 3)$ and its out-neighbors form a $\overrightarrow{TT}_3$, and so there are at least two copies of $\overrightarrow{TT}_3$ in $T$. Therefore, there are at least two copies of $\overrightarrow{TT}_3$ in $T$ and $z(\overrightarrow{TT}_3)\leq 4$. Now consider the directed cycle $\overrightarrow{C}_3$, clearly, there is no homomorphism from $\overrightarrow{TT}_3$ to $\overrightarrow{C}_3$, which means that $z(\overrightarrow{TT}_3)\geq 4$. Hence, $z(\overrightarrow{TT}_3)=4$.

\noindent{\bf Part (b)} We first consider the case that $n=4$. If $|E(\overrightarrow{G})|=|E(T(n,3))|+1$, then $\overrightarrow{G}$ is a tournament and so has at least two copies of $\overrightarrow{TT}_3$ by the argument of (a).

Next we consider the case that $n=5$. Let $\overrightarrow{G}$ be an oriented graph on $\{v_i\mid i\in [5]\}$ with $|E(T(5,3))|+1=9$ arcs. Clearly, $\overrightarrow{G}$ contains two non-adjacent vertices, say $v_1,v_2$. Observe that the induced oriented subgraph $\overrightarrow{G}[\{v_1,v_3,v_4,v_5\}]$ is a 4-tournament, and so contains at least two copies of $\overrightarrow{TT}_3$, moreover, one of them must contain $v_1$, we denote this copy by $\overrightarrow{T}'$. Now $\overrightarrow{G}'=\overrightarrow{G}-v_1$ is also a 4-tournament, so it contains at least two copies of $\overrightarrow{TT}_3$, combining with $\overrightarrow{T}'$, we get at least three copies of $\overrightarrow{TT}_3$ in $\overrightarrow{G}$, as desired.

Finally, it remains to consider the case that $n=6$. Let $\overrightarrow{G}$ be an oriented graph on $\{v_i\mid i\in [6]\}$ with $|E(T(6,3))|+1=13$ arcs. Observe that $\overrightarrow{G}$ contains two couples of non-adjacent vertices, say $\{v_1, v_2\}$ and $\{a,b\}$. 

CASE 1. $|\{v_1,v_2\}\cap\{a,b\}|=1$. Without loss of generality, assume that $a=v_2, b=v_3$. Observe that $\overrightarrow{G}[\{v_3,v_4,v_5,v_6\}]$ is a 4-tournament, and so contains at least two copies of $\overrightarrow{TT}_3$, and one of them must contain $v_3$, we denote this copy by $\overrightarrow{T}_1$. 
Now $\overrightarrow{G}_1=\overrightarrow{G}-v_3$ has $5$ vertices and $9$ arcs. According to the argument of the case that $n=5$, $\overrightarrow{G}_1$ contains at least three copies of $\overrightarrow{TT}_3$, combining with $\overrightarrow{T}
_1$, we get at least four copies of $\overrightarrow{TT}_3$ in $\overrightarrow{G}$, as desired.

CASE 2. $|\{v_1,v_2\}\cap\{a,b\}|=0$. Without loss of generality, assume that $a=v_3, b=v_4$. Observe that $\overrightarrow{G}_3[\{v_2,v_4,v_5,v_6\}]$ is a 4-tournament, and so contains at least two copies of $\overrightarrow{TT}_3$, and one of them must contain $v_4$, we denote this copy by $\overrightarrow {T}_1$. 
Let $\overrightarrow{G}_1= \overrightarrow{G}-v_4$. Now $\overrightarrow{G}_1$ has $5$ vertices and $9$ arcs. According to the argument of the case that $n=5$, $\overrightarrow{G}_1$ contains at least three copies of $\overrightarrow{TT}_3$, combining with $\overrightarrow{T}
_1$, we get at least four copies of $\overrightarrow{TT}_3$ in $\overrightarrow{G}$, as desired.
\end{proof}


Now we are ready to prove Theorem~\ref{TT3}.

\vskip 0.5cm

\noindent
{\bf Theorem~\ref{TT3}:} {\em Let $\overrightarrow{G}$ be an oriented graph with order $n=3k+t>3$, where $k$ is a positive integer and $t\in\{0,1,2\}$. If  $|E(\overrightarrow{G})|=|E(T(n,3))|+1$, then there are at least     $f(k)$ copies of $\overrightarrow{TT}_3$ in $\overrightarrow{G}$, where $f(k)=\begin{cases}
       2k,& t\in \{0,1\}\\
       2k+1,&t=2\\
    \end{cases}$.}

\vskip 0.5cm


\begin{proof}
We prove the theorem by induction on $n$. It holds when $n\in \{4,5,6\}$ by Proposition~\ref{smallgraph}. Assume that it holds for any oriented graph with order $n-1$, where $n\geq 7$, and now we will prove that it also holds for an oriented graph $\overrightarrow{G}$ with order $n$. The proof is divided into the following three cases: 

    CASE 1. $n=3k$. Observe that $\overrightarrow{G}$ contains at least one vertex of total degree at most $2k$; otherwise, we have $$|E(\overrightarrow{G})|\geq \frac{1}{2}(2k+1)3k=3k^2+\frac{3}{2}k>3k^2+1=|E(T(3k,3))|+1,$$ a contradiction. 
    
    CASE 1.1. There is a vertex $v$ with $d(v)\leq 2k-1$ in $\overrightarrow{G}$. Now $\overrightarrow{G}_1=\overrightarrow{G}-v$ has $3k-1$ vertices and at least $$3k^2-2k+2=|E(T(3k-1,3))|+2=|E\big(T\big(3(k-1)+2,3\big)\big)|+2$$ arcs. By the induction hypothesis, $\overrightarrow{G}_1$ contains at least one copy, say $T$, of $\overrightarrow{TT}_3$. Let $e\in A(T).$ The oriented graph $\overrightarrow{G}_2=\overrightarrow{G}_1-e$ has $3k-1$ vertices and at least $$3k^2-2k+1=|E(T(3k-1,3))|+1=|E
\big(T\big(3(k-1)+2,3\big)\big)|+1$$ arcs. By the induction hypothesis, $\overrightarrow{G}_2$ contains    
    at least $2(k-1)+1=2k-1$ copies of $\overrightarrow{TT}_3$, combining with $T$, we get at least $2k$ copies of $\overrightarrow{TT}_3$ in $\overrightarrow{G}_1$ (and so in $\overrightarrow{G}$), as desired.

    CASE 1.2. $d(v)\geq 2k$ for each $v\in V(\overrightarrow{G})$. We use $x$ to denote the number of vertices in $\overrightarrow{G}$ with total degree exactly $2k$, and so
    $$|E(\overrightarrow{G})|\geq \frac{1}{2}\big(2kx+(3k-x)(2k+1)\big)=3k^2+\frac{3}{2}k-\frac{1}{2}x.$$
    If $x\leq 3k-3$, then $$|E(\overrightarrow{G})|\geq 3k^2+\frac{3}{2}>3k^2+1=|E(T(3k,3))|+1,$$
    a contradiction. Therefore, the number of vertices in $\overrightarrow{G}$ with total degree exactly $2k$ is at least $3k-2$. Choose one such vertex $u$. The oriented graph $\overrightarrow{G}_1=\overrightarrow{G}-u$ has $3k-1$ vertices and $$3k^2-2k+1=|E(T(3k-1,3))|+1=|E
\big(T\big(3(k-1)+2,3\big)\big)|+1$$ arcs. By the induction hypothesis, $\overrightarrow{G}_1$ contains at least one copy, say $T$, of $\overrightarrow{TT}_3$. Furthermore, $T$ contains at least one vertex $u'$ with total degree $2k$. The oriented graph $\overrightarrow{G}_2=\overrightarrow{G}-u'$ also has $3k-1$ vertices and $$|E(T(3k-1,3))|+1=|E
\big(T\big(3(k-1)+2,3\big)\big)|+1$$ arcs. By the induction hypothesis, $\overrightarrow{G}_2$ contains at least $2(k-1)+1=2k-1$ copies of $\overrightarrow{TT}_3$, combining with $T$, we get  at least $2k$ copies of $\overrightarrow{TT}_3$ in $\overrightarrow{G}$, as desired.
    
    CASE 2. $n=3k+1$. Observe that $\overrightarrow{G}$ contains at least one vertex of total degree at most $2k$; otherwise, 
    $$|E(\overrightarrow{G})|\geq \frac{1}{2}(2k+1)(3k+1)=3k^2+\frac{5}{2}k+\frac{1}{2}>3k^2+2k+1=|E(T(3k+1,3))|+1,$$a contradiction.
    Let $\overrightarrow{G}_1$ be the oriented graph obtained from $\overrightarrow{G}$ by removing a vertex of total degree at most $2k$. Now $\overrightarrow{G}_1$ has $3k$ vertices and at least
    $$3k^2+1=|E(T(3k,3))|+1$$
    arcs. Hence, by the induction hypothesis, $\overrightarrow{G}_1$ contains at least $2k$ copies of $\overrightarrow{TT}_3$, as desired.

    CASE 3. $n=3k+2$. Observe that $\overrightarrow{G}$ contains at least one vertex of total degree at most $2k+1$; otherwise, 
    $$|E(\overrightarrow{G})|\geq\frac{1}{2}(2k+2)(3k+2)=3k^2+5k+2>3k^2+4k+2=|E(T(3k+2,3))|+1=|E(\overrightarrow{G})|,$$
    a contradiction. 
    
    CASE 3.1. There is a vertex $v$ with $d(v)\leq 2k$ in $\overrightarrow{G}$. 
    Now $\overrightarrow{G}_1=\overrightarrow{G}-v$ has $3k+1$ vertices and at least $$|E(T(3k+2,3))|+1-2k=3k^2+2k+2=|E(T(3k+1,3))|+2$$ arcs. By the induction hypothesis, $\overrightarrow{G}_1$ contains at least one copy, say $T$, of $\overrightarrow{TT}_3$. Let $e\in A(T)$. The oriented graph $\overrightarrow{G}_2=\overrightarrow{G}_1-e$ has $3k+1$
    vertices and at least $|E(T(3k+1,3))|+1$ arcs. By the induction hypothesis, 
    $\overrightarrow{G}_2$ contains at least $2k$ copies of $\overrightarrow{TT}_3$, combining with $T$, we get at least $2k+1$ copies of $\overrightarrow{TT}_3$ in $\overrightarrow{G}_1$ (and so in $\overrightarrow{G}$), as desired.

    CASE 3.2. $d(v)\geq 2k+1$ for each $v\in V(\overrightarrow{G})$. 
    We use $x$ to denote the number of vertices in $\overrightarrow{G}$ with total degree exactly $2k+1$, and so
    $$|E(\overrightarrow{G})|\geq\frac{1}{2}\big(x(2k+1)+(3k+2-x)(2k+2)\big)=3k^2+5k-\frac{x}{2}+2.$$
    If $x\leq 2k-1$, then 
    $$|E(\overrightarrow{G})|\geq 3k^2+4k+\frac{5}{2}>3k^2+4k+2=|E(T(3k+2,3))|+1,$$
    a contradiction. Therefore, the number of vertices in $\overrightarrow{G}$ with total degree exactly $2k+1$ is at least $2k$.

    CASE 3.2.1. $d(v)\leq 2k+2$ for each $v\in V(\overrightarrow{G})$. 
    We use $V_1$ and $V_2$ to denote the set of vertices with total degrees $2k+1$ and $2k+2$ in $\overrightarrow{G}$, respectively. As $|E(\overrightarrow{G})|=3k^2+4k+2$, we can deduce that $|V_1|=2k, |V_2|=k+2$. We now prove the following claim:
    
    
    \begin{claim}\label{claim1}
    Let $u$ be a vertex with total degree $2k+1$ in $\overrightarrow{G}$. There exist two vertices, say $v, w$, of $\overrightarrow{G}$ such that $d(v)=2k+1, d(w)=2k+2$, and $\{u,v,w\}$ induces a triangle in the underlying graph of $\overrightarrow{G}$. 
    \end{claim}
    \begin{proof}
Since $d(u)=2k+1$ and $|V_2|=k+2$, there are at least $k-1$ neighbors of $u$ with total degree $2k+1$ in $\overrightarrow{G}$. Denote this set of vertices by $N_1(u)$, so $|N_1(u)|\geq k-1$. Suppose that there is no vertex of total degree exactly $2k+2$ which is adjacent to $u$ and its neighbors with total degree $2k+1$. Since  $d(u)=2k+1$ and $|V_1|=2k$, there exists a vertex $w'\in V_2$  which is adjacent to $u$. Moreover, there are at least $k+1$ vertices of total degree $2k+1$ which are adjacent to $w'$. Denote this set of vertices by $N_1(w')$, so $|N_1(w')|\geq k+1$. Hence, $u\in N_1(w')$ and $N_1(u)\cap N_1(w')=\emptyset$ (otherwise, setting $v\in N_1(u)\cap N_1(w')$ and $w=w'$ yields the desired vertices). Since $d(u)=2k+1$, $u$ is adjacent to each vertex of $V(\overrightarrow{G})\setminus N_1(w')$ and $|V(\overrightarrow{G})\setminus N_1(w')|=2k+1$. In particular, $u$ is adjacent to all vertices of $V_2$.   
Observe that there are no arcs between $V_1\setminus N_1(w')$ and $V_2$ (otherwise, let $w''\in V_2$ be adjacent to $v'\in V_1\setminus N_1(w')$, then setting $v'=v$ and $w''=w$ yields the desired vertices). However, now the total degree of each vertex in $V_1\setminus N_1(w')$ is at most $2k-1$, a contradiction. Hence, there is a vertex of total degree exactly $2k+2$ which is adjacent to $u$ and its neighbors with total degree $2k+1$, as desired.
    \end{proof}

Let $u$ be  a vertex with total degree $2k+1$ in $\overrightarrow{G}$. By Claim~\ref{claim1}, there exists two vertices, say $v, w$, of $\overrightarrow{G}$ such that $d(v)=2k+1, d(w)=2k+2$, and $\{u,v,w\}$ induces a triangle in the underlying graph of $\overrightarrow{G}$.


Let $\overrightarrow{G}_1=\overrightarrow{G}-\{u,v\}$. Clearly, $\overrightarrow{G}_1$ contains $3k$ vertices and $$|E(T(3k+2,3))|+1-2k-1-2k=3k^2+1=|E(T(3k,3))|+1$$ arcs. By the induction hypothesis, $\overrightarrow{G}_1$ contains at least $2k$ copies of $\overrightarrow{TT}_3$ and thus $\overrightarrow{G}$ contains at least $2k$ copies of $\overrightarrow{TT}_3$. Therefore, it suffices to find one additional copy of $\overrightarrow{TT}_3$ in $\overrightarrow{G}$ containing $u$ or $v$. Now we partition $V(\overrightarrow{G}_1)$ into four parts as follows: let $Q_1$ denote the set of vertices adjacent to $u$ but not $v$ in $\overrightarrow{G}$, $Q_2$ denote the set of vertices adjacent to $v$ but not $u$ in $\overrightarrow{G}$, $Q_3$ denote the set of vertices adjacent to both $u$ and $v$  in $\overrightarrow{G}$, and $Q_4=V(\overrightarrow{G}_1)\setminus(Q_1\cup Q_2\cup Q_3)$. Clearly, $w\in Q_3$ and $|Q_3|\geq k$ by Claim~\ref{claim1},. 
    
Recall that every 4-tournament contains at least two copies of $\overrightarrow{TT}_3$. If there exists an arc $\overrightarrow{w_1w_2}$ in $\overrightarrow{G}[Q_3]$, then $\overrightarrow{G}[\{u,v,w_1,w_2\}]$ is a 4-tournament, and so contains at least two copies of $\overrightarrow{TT}_3$, and one of them must contain $u$ or $v$, as desired. In the following, it remains to show that there is an arc in $\overrightarrow{G}[Q_3]$.
       
       CASE 3.2.1.1. $|Q_4|\geq 1$. Since $n=3k+2, d(u)=d(v)=2k+1$, and $u$ and $v$ are adjacent in $\overrightarrow{G}$,  we have $|Q_3|\geq k+1$, which implies that $|V(\overrightarrow{G})\backslash Q_3|\leq 2k+1$, and thus $w$ is adjacent to at least one vertex in $Q_3$ as $d(w)=2k+2$, as desired. 

       CASE 3.2.1.2. $Q_4=\emptyset$. We first consider the case that $|Q_3|\geq k+1$. Now $|V(\overrightarrow{G})\setminus Q_3|\leq 2k+1$. Since $d(w)=2k+2$,  there must be a vertex $u_3\in Q_3$ which is adjacent to $w$. Then $\overrightarrow{G}[\{u,v,w,u_3\}]$ is a 4-tournament, which contains the desired copy of $\overrightarrow{TT}_3$. We next consider the case that $|Q_3|=k$. In order to ensure that $Q_3$ contains no arcs, $w$ must be adjacent to every vertex in $Q_1$ and $Q_2$. Consequently, there are no arcs within $\overrightarrow{G}[Q_1]$ or $\overrightarrow{G}[Q_2]$. Indeed, without loss of generality, assume that there exists an arc $\overrightarrow{u_1u_2}$ in $\overrightarrow{G}[Q_1]$. Now $\overrightarrow{G}[\{u_1,u_2,u,w\}]$ is a 4-tournament, which contains at least two copies of $\overrightarrow{TT}_3$ and one of them contains $u$. 
       Therefore, there is only one remaining case: $|Q_3|=k,~|Q_1|=d(u)-1-|Q_3|=d(v)-1-|Q_3|=k=|Q_2|$ and no arcs exist within $Q_1,Q_2$ or $Q_3$. However, now $$|E(\overrightarrow{G})|\leq 3k^2+2k+1+2k=3k^2+4k+1< |E(T(3k+2,3))|+1,$$ a contradiction. 
       
       CASE 3.2.2. There exists a vertex $u'\in  V(\overrightarrow{G})$ with $d(u')>2k+2$. We
        use $x$ to denote the number of vertices in $\overrightarrow{G}$ with total degree exactly $2k+1$. So now
    $$|E(\overrightarrow{G})|\geq\frac{1}{2}\big(x(2k+1)+(3k+1-x)(2k+2)+2k+3\big)=3k^2+5k-\frac{x}{2}+\frac{5}{2}.$$
       If $x\leq 2k$, then
       $$|E(\overrightarrow{G})|\geq 3k^2+4k+\frac{5}{2}>3k^2+4k+2=|E(T(3k+2,3))|+1,$$ a contradiction. Therefore, the number of vertices in $\overrightarrow{G}$ with total degree exactly $2k+1$ is at least $2k+1$. We use $V_1$ to denote the set of vertices with total degree $2k+1$, so $|V_1|\geq 2k+1$. Now we prove the following claim.
       \begin{claim}\label{claim2}
           There exist two adjacent vertices, say $u$, $v$, in $V_1$, such that $\{u,v,u'\}$ induces a triangle in the underlying graph of $\overrightarrow{G}$.
       \end{claim}
\begin{proof}
Since $n=3k+2, |V_1|\geq 2k+1$ and $d(u')>2k+2$, there are 
    at least $k+3$ neighbors of $u'$ with total degree $2k+1$ in $\overrightarrow{G}$. Denote this set of vertices by $N_1(u')$. Let $u_1\in N_1(u')$. Since $d(u_1)=2k+1$, there are at least $k$ neighbors of $u_1$ with total degree $2k+1$ in $\overrightarrow{G}$. Denote this set of vertices by $N_1(u_1)$. If $N_1(u_1)\cap N_1(u')\neq \emptyset$, then we get the desired vertices by setting $u=u_1$ and $v\in N_1(u_1)\cap N_1(u')$. Otherwise, we have $$d(u_1)\leq (3k+2)-|N_1(u')|\leq 3k+2-(k+3)=2k-1,$$ a contradiction. This completes the proof of Claim~\ref{claim2}.
\end{proof}
Let $\overrightarrow{G}_1=\overrightarrow{G}-\{u,v\}$. Clearly, $\overrightarrow{G}_1$ contains $3k$ vertices and $$|E(T(3k+2,3))|+1-2k-1-2k=3k^2+1=|E(T(3k,3))|+1$$ arcs. By the induction hypothesis, $\overrightarrow{G}_1$ contains at least $2k$ copies of $\overrightarrow{TT}_3$ and thus $\overrightarrow{G}$ contains at least $2k$ copies of $\overrightarrow{TT}_3$. Therefore, it suffices to find one additional copy of $\overrightarrow{TT}_3$ in $\overrightarrow{G}$ containing $u$ or $v$.
By Claim~\ref{claim2}, there exist two adjacent vertices, say $u$, $v$, in $V_1$, such that $\{u,v,u'\}$ induces a triangle in the underlying graph of $\overrightarrow{G}$. We define the sets $Q_1,Q_2,Q_3,Q_4$ similarly as those in Case 3.2.1. Now $u'\in Q_3$. Since $d(u)=d(v)=2k+1$, they have at least $k$ common neighbors.  Consequently, $|V(G)\backslash Q_3|\leq2k+2$, which implies that $u'$ is adjacent to at least one vertex $u''$ in $Q_3$, then $\overrightarrow{G}[\{u,v,u',u''\}]$ is a 4-tournament, and so contains at least two copies of $\overrightarrow{TT}_3$, and one of them must contain $u$ or $v$, as desired. This completes the proof.
\end{proof}

\section{Proofs of Theorems~\ref{OMM} and~\ref{TTdensity}}


The following result can be seen as an oriented version of Moon-Moser inequality~\cite{Moon-Moser}.

\vskip 0.5cm

\noindent
{\bf Theorem~\ref{OMM}:} {\em   For any oriented graph $\overrightarrow{G}$ with order $n$, denote by $N_r$ and $M_r$ the numbers of $r$-tournaments and copies of $\overrightarrow{TT}_r$ contained in $\overrightarrow{G}$, respectively. The following inequalities hold:
\begin{description}
\item[(a)]$\frac{N_{r+1}}{N_{r}}\geq \frac{1}{r^2-1}\left(r^2\frac{N_{r}}{N_{r-1}}-n\right)$ for $r\geq 2$;
\item[(b)]$\frac{M_{r+1}}{M_{r}}\geq \frac{1}{r^2-1}\left(r^2\frac{M_{r}}{M_{r-1}}-n\right)$ for $r\geq 3$.
\end{description}
\vskip 0.5cm

\begin{proof}
\noindent{\bf Part (a)}
Let $\overrightarrow{T}$ be an $r$-tournament in $\overrightarrow{G}$ and $\overrightarrow{R}$ be an $r$-vertex induced subdigraph of $\overrightarrow{G}$ which is not a tournament, such that $|\overrightarrow{T}\cap \overrightarrow{R}|=r-1$.  
We now do double counting on the number of pairs $(\overrightarrow{T},\overrightarrow{R})$, denoted by $P$, in $\overrightarrow{G}$.

On the one hand, let $\{\overrightarrow{H}_i\mid i\in [N_{r-1}]\}$ be the set of all $(r-1)$-tournaments in $\overrightarrow{G}$. For each $i\in [N_{r-1}]$, let $a_i$ be the number of $r$-tournaments in $\overrightarrow{G}$ containing $\overrightarrow{H}_i$ i.e. $a_i$ is the number of vertices $v$ in $\overrightarrow{G}$ such that $V(\overrightarrow{H}_i)\cup \{v\}$ induces an $r$-tournament in $\overrightarrow{G}$. It can be checked that
$$\sum_{i=1}^{N_{r-1}}a_i=rN_r.$$
Also observe that there exist $n-(r-1)-a_i=n-r+1-a_i$ vertices $v' \in V(\overrightarrow{G})$ such that $V(\overrightarrow{H}_i)\cup \{v'\}$ is an $r$-vertex subset which induces an $\overrightarrow{R}$. Therefore,
\begin{equation*}
\begin{aligned}
P&=\sum_{i=1}^{N_{r-1}}a_i(n-r+1-a_i)\\
 &=(n-r+1)\sum_{i=1}^{N_{r-1}}a_i-\sum_{i=1}^{N_{r-1}}{a_i}^2\\
 &=(n-r+1)rN_r-\sum_{i=1}^{N_{r-1}}{a_i}^2\\
 &\leq (n-r+1)rN_r-\frac{(\sum_{i=1}^{N_{r-1}}a_i\cdot 1)^2}{N_{r-1}}\\
 &=(n-r+1)rN_r-\frac{r^2{N_r}^2}{N_{r-1}},
\end{aligned}
\end{equation*}
where the inequality holds by Jensen's inequality.

On the other hand, let $\{\overrightarrow{Q}_j\mid j\in [N_{r}]\}$ be the set of all $r$-tournaments in $\overrightarrow{G}$. For each $j\in [N_{r}]$, let $b_j$ be the number of $(r+1)$-tournaments in $\overrightarrow{G}$ containing $\overrightarrow{Q}_j$ i.e. $b_j$ is the number of vertices $u$ in $\overrightarrow{G}$ such that $V(\overrightarrow{Q}_j)\cup \{u\}$ induces an $(r+1)$-tournament in $\overrightarrow{G}$. It can be checked that
$$\sum_{j=1}^{N_{r}}b_j=(r+1)N_{r+1}.$$
Clearly, there exist $n-r-b_j$ vertices $u' \in V(\overrightarrow{G})$ such that the $(r+1)$-vertex subset $V(\overrightarrow{Q}_j)\cup \{u'\}$ does not induce an $(r+1)$-tournament in $\overrightarrow{G}$. Hence, there exists at least one vertex $y\in V(\overrightarrow{Q}_j)$ such that the union of $u',~y$ and any $r-2$ vertices of $\overrightarrow{Q}_j\setminus \{y\}$ induces an $\overrightarrow{R}$ in $\overrightarrow{G}$. Therefore,
\begin{equation*}
\begin{aligned}
P&\geq \sum_{j=1}^{N_r}(n-r-b_j)(r-1)\\
&=(r-1)\big(\sum_{j=1}^{N_r}(n-r)-\sum_{j=1}^{N_r}b_j\big)\\
&=(r-1)\big(N_r(n-r)-(r+1)N_{r+1}\big).
\end{aligned}
\end{equation*}
Hence,
$$(r-1)\big(N_r(n-r)-(r+1)N_{r+1}\big)\leq P\leq (n-r+1)rN_r-\frac{r^2{N_r}^2}{N_{r-1}},$$ and so $$(r+1)(r-1)N_{r+1}\geq N_r(n-r)(r-1)-(n-r+1)rN_r+\frac{r^2{N_r}^2}{N_{r-1}}.$$
Now we have
\begin{equation*}
\begin{aligned}
N_{r+1}&\geq N_r\frac{n-r}{r+1}-\frac{(n-r+1)r}{(r+1)(r-1)}N_r+\frac{r^2N_r}{(r+1)(r-1)}\frac{N_r}{N_{r-1}}\\
&= N_r\left(\frac{(n-r)(r-1)-(n-r+1)r}{(r+1)(r-1)} \right)+\frac{r^2}{r^2-1}\frac{N_r}{N_{r-1}}\\
&=\frac{r^2N_r}{r^2-1}\left(\frac{N_r}{N_{r-1}}-\frac{n}{r^2}\right),
\end{aligned}
\end{equation*}
that is, $$\frac{N_{r+1}}{N_{r}}\geq \frac{1}{r^2-1}\left(r^2\frac{N_{r}}{N_{r-1}}-n\right),$$
as desired.

\noindent{\bf Part (b)} When $r\geq 3$, we apply the same double counting method as in \noindent{\bf Part (a)}, by replacing ``$r$-tournament'' with ``$\overrightarrow{TT}_r$'' and ``$N_r$'' with ``$M_r$''. And then the desired inequality follows directly. 
\end{proof}

\textnormal{Note that the inequality in Theorem~\ref{OMM}(b) does not hold when $r=2$: let $\{\overrightarrow{Q}_j\mid j\in [M_{2}]\}$ be the set of all transitive tournaments with order $2$ (now each such tournament is an arc) in $\overrightarrow{G}$. For any $\overrightarrow{Q}_j$, there may exist some vertices $u'\in V(\overrightarrow{G})$ such that the $3$-vertex subset $V(\overrightarrow{Q}_j)\cup\{u'\}$ induces a directed cycle rather than a $\overrightarrow{TT}_3$ in $\overrightarrow{G}$. However, for any vertex $y\in V(\overrightarrow{Q}_j)$, the union of $u'$ and $y$ induces a $\overrightarrow{TT}_2$ rather than the required non-transitive subgraph $\overrightarrow{R}$ in $\overrightarrow{G}$. Therefore, the lower bound of $P$ does not hold now.}

\textnormal{We are now in a position to prove Theorem~\ref{TTdensity}. Recall the theorem.}
\vskip 0.5cm

\noindent
{\bf Theorem~\ref{TTdensity}:} {\em   Let $\overrightarrow{G}$ be an oriented graph with order $n$.
If $$|E(\overrightarrow{G})|\geq(1-\frac{1}{t})\frac{n^2}{2}$$ for some $t\in \mathbb{R}^+$, then $$N_r(\overrightarrow{G})\geq \binom{t}{r}(\frac{n}{t})^r,$$ where $$\binom{t}{r}= \left\{\begin{matrix}
\frac{t(t-1)\cdots (t-r+1)}{r!}, & t>r-1, \\
0, & t\leq r-1.
\end{matrix}\right.$$}

\vskip 0.5cm

\begin{proof}
We prove the result by induction on $r$. 
Since $$N_1(\overrightarrow{G})=n=\binom{t}{1}\big(\frac{n}{t}\big)$$ and $$N_2(\overrightarrow{G})=|E(\overrightarrow{G})|\geq(1-\frac{1}{t})\frac{n^2}{2}=\binom{t}{2}(\frac{n}{t})^2,$$
the result holds when $r\in \{1,2\}$. 



Now we assume that the result holds when $r\leq k-1$, and consider the case that $r=k$. We need the following claim:

\begin{claim}
For each $r\geq 1$, we have
$$d_r=\frac{N_{r+1}(\overrightarrow{G})}{N_r(\overrightarrow{G})}\geq \frac{n(t-r)}{t(r+1)}.$$
\end{claim}
\begin{proof} We prove the claim by induction on $r$. The case that $r=1$ holds, as $$d_1=\frac{N_2(\overrightarrow{G})}{N_1(\overrightarrow{G})}\geq\frac{(1-\frac{1}{t})\frac{n^2}{2}}{n}=\frac{n(t-1)}{2t}.$$ Assume that the result holds when $2\leq r\leq k-1$, and consider the case that $r=k$. By Theorem~\ref{OMM} and the induction hypothesis, we have



\begin{equation*}
\begin{aligned}
d_r=d_k=\frac{N_{k+1}(\overrightarrow{G})}{N_{k}(\overrightarrow{G})}&\geq \frac{1}{k^2-1}\left(\frac{k^2N_{k}}{N_{k-1}}-n\right)\\
&\geq \frac{1}{k^2-1}\left(k^2\cdot\frac{n(t-k+1)}{tk}-n\right)\\
&=\frac{n}{t}\cdot \frac{kt-k^2+k-t}{k^2-1}\\
&=\frac{n(t-k)}{t(k+1)}.
\end{aligned}
\end{equation*}
Therefore, the claim holds.
\end{proof}

By the above claim and the induction hypothesis, we have
\begin{equation*}
\begin{aligned}
N_{k+1}(\overrightarrow{G})&\geq N_k(\overrightarrow{G})\cdot {\frac{n(t-k)}{t(k+1)}}\\
&\geq \binom{t}{k}(\frac{n}{t})^k\cdot {\frac{n(t-k)}{t(k+1)}}\\
&=\binom{t}{k+1}(\frac{n}{t})^{k+1}.
\end{aligned}
\end{equation*}
Hence, the result holds.
\end{proof}

\section{Proof of Theorem~\ref{antidirected4-cycle}}

\textnormal{For an oriented graph $\overrightarrow{G}$ with order $n$, we use $d^+(\overrightarrow{G})=\frac{|E(\overrightarrow{G})|}{n}$ to denote the {\em average out-degree} of $\overrightarrow{G}$.} 

\vskip 0.5cm

\noindent
{\bf Theorem~\ref{antidirected4-cycle}:} {\em   Let $s, t\geq 1$ be two positive integers. For every $n$-vertex oriented graph $\overrightarrow{G}$ with sufficiently large $n$, if $|E(\overrightarrow{G})|\geq es^{\frac{1}{t}}n^{2-\frac{1}{t}}$, then there are at least $\big(\frac{e}{t}\big)^t n^t$ copies of $\overrightarrow{K}_{s,t}$ in $\overrightarrow{G}$, where $e$ denotes Euler's number.}

\vskip 0.5cm

\begin{proof}Let $d^+=d^+(\overrightarrow{G})$ and $d^+(v)$ be the out-degree of a vertex $v\in V(\overrightarrow{G})$.
We use $\alpha$ to denote the number of $\overrightarrow{K}_{1,t}$ in $\overrightarrow{G}$. 
Observe that
$$\alpha\geq\sum_{v\in V(\overrightarrow{G})}\binom{d^+(v)}{t}\geq n \binom{\frac{1}{n}\sum_{v\in V(\overrightarrow{G})}d^+(v)}{t}=n \binom{\frac{1}{n}{|E(\overrightarrow{G})|}}{t}=n\binom{d^+}{t},$$
where the second inequality follows from Jensen's inequality.

Let $d^-(u_1,u_2,\ldots,u_t)$ denote the common in-degree of the vertices $u_1,u_2,\ldots,u_t\in V(\overrightarrow{G})$. We use $\beta$ to denote the average common in-degree of a $t$-set of vertices in $\overrightarrow{G}$. Observe that
$$\beta=\frac{1}{\binom{n}{t}}\cdot \sum_{u_1,u_2,\ldots,u_t\in \binom{V(\overrightarrow{G})}{t}}d^-(u_1,u_2,\ldots,u_t)=\frac{\alpha}{\binom{n}{t}}\geq \frac{n\binom{d^+}{t}}{\binom{n}{t}}.$$
So the number of $\overrightarrow{K}_{s,t}$ in $\overrightarrow{G}$ is at least
\begin{equation*}
\begin{aligned}
\sum_{u_1,u_2,\ldots,u_t\in \binom{V(\overrightarrow{G})}{t}}\binom{d^-(u_1,u_2,\ldots,u_t)}{s}&\geq\binom{n}{t}\binom{\beta}{s}\geq\binom{n}{t} \binom{\frac{n\binom{d^+}{t}}{\binom{n}{t}}}{s},
\end{aligned}
\end{equation*}
where the first inequality follows from Jensen's inequality.
Observe that
$$\frac{n\binom{d^+}{t}}{\binom{n}{t}}\geq\frac{n\binom{d^+}{t}}{(\frac{en}{t})^t}\geq \frac{n(\frac{d^+}{t})^t}{(\frac{en}{t})^t}\geq \frac{n(\frac{e}{t}s^{\frac{1}{t}}n^{1-\frac{1}{t}})^t}{(\frac{e}{t})^tn^t}\geq  s,$$
where the first inequality follows from Stirling's inequality, which states that
$$\binom{m}{k}\leq (\frac{em}{k})^k$$
for any $k\in[m]$, and the second inequality follows from  $\binom{m}{k}\geq (\frac{m}{k})^k$ for any $k\in[m]$.\\
Therefore,
\begin{equation*}
\begin{aligned}
\binom{n}{t} \binom{\frac{n\binom{d^+}{t}}{\binom{n}{t}}}{s}&\geq \binom{n}{t} \bigg(\frac{n\binom{d^+}{t}}{s\binom{n}{t}}\bigg)^s\geq \binom{n}{t} \bigg(\frac{n(\frac{d^+}{t})^t}{s\binom{n}{t}}\bigg)^s=\frac{(\frac{d^+}{t})^{st}n^s}{s^s\binom{n}{t}^{s-1}}\\
&\geq \frac{(\frac{d^+}{t})^{st}n^s}{s^s(\frac{en}{t})^{st-t}}=
\frac{(d^+)^{st}n^s}{s^st^t {(e n)}^{st-t}}\geq \frac{(es^{\frac{1}{t}}n^{1-\frac{1}{t}})^{st}n^s}{ s^st^t {(e n)}^{st-t}}\\
&=\big(\frac{e}{t}\big)^t n^t,
\end{aligned}
\end{equation*}
where the first and second inequalities follow from $\binom{m}{k}\geq (\frac{m}{k})^k$ for any $ k\in [m]$ and the third inequality follows from Stirling's inequality.\\
Hence,
$$\sum_{u_1,u_2,\ldots,u_t\in \binom{V(\overrightarrow{G})}{t}}\binom{d^-(u_1,u_2,\ldots,u_t)}{s}\geq \big(\frac{e}{t}\big)^t n^t.$$
We finish the proof.
\end{proof}




\vskip 1cm

\noindent {\bf Acknowledgement} \textnormal{This work was supported by National Natural Science Foundation of China under Grant No. 12371352 and Yongjiang Talent Introduction Programme of Ningbo under Grant No. 2021B-011-G.}

\vskip 3mm
\noindent {\bf Data availability}\textnormal{ No data was used for the research described in the article.}


\begin{thebibliography}{99}
\normalfont
\bibitem{ahlrt} L. Addario-Berry, F. Havet, C. Linhares Sales, B. Reed and S. Thomass\'e,
Oriented trees in digraphs, Discrete Math., 313, 2013, 967--974.

\bibitem{Alon-Krivelevich-Sudakov}
N. Alon, M. Krivelevich and B. Sudakov, Tur\'{a}n numbers of bipartite graphs and related Ramsey-type questions, Combin. Prob. Comput., 12, 2003, 477--494.

\bibitem{bh} W.G. Brown and F. Harary, Extremal digraphs, Combinatorial theory and its applications, Colloq. Math. Soc. J. Bolyai, 4, 1969, 135--198.

\bibitem{burr} S.A. Burr, Subtrees of directed graphs and hypergraphs, In: Proc. 11th South eastern Conf. Combinatorics, Graph Theory and Computing, Florida Atlantic Univ., Boca Raton, Fla. I Vol. 28, 1980, 227--239.

\bibitem{Erdos1962}
P. Erd\H{o}s, On a theorem of Rademacher-Tur\'{a}n, Illinois J. Math., 6, 1962, 122--127.

\bibitem{Erdos-Simonovits}
P. Erd\H{o}s and M. Simonovits, A limit theorem in graph theory, Studia Sci. Math. Hungar., 1, 1966, 51--57.

\bibitem{Erdos-Simonovits1984}
P. Erd\H{o}s and M. Simonovits, Cube-supersaturated graphs and related problems, Progress in graph theory (Waterloo, Ont., 1982), pages 203--218, 1984.

\bibitem{Erdos-Stone}
P. Erd\H{o}s and A. H. Stone, On the structure of linear graphs, Bull. Amer. Math. Soc., 52, 1946, 1087--1091.

\bibitem{Furedi}
Z. F\"{u}redi, On a Tur\'{a}n type problem of Erd\H{o}s, Combinatorica, 11(1), 1991, 75--79.

\bibitem{Gerbner-Hu-Sun} D. Gerbner, X. Hu and Y. Sun, On oriented Turán problems, arXiv:2602.04324.

\bibitem{gra} R.L. Graham, On subtrees of directed graphs with no path of length exceeding one, Canad. Math. Bull., 13, 1970, 329--332.

\bibitem{gs} A. Grzesik and M. Skrzypczyk, Antidirected paths in oriented graphs, arXiv:2506.11866.

\bibitem{Katona-Nemetz-Simonovits}
G. Katona, T. Nemetz and M. Simonovits, On a problem of Tur\'{a}n in the theory of graphs, Mat. Lapok, 15, 1964, 228--238.

\bibitem{ks} T. Klimo\u{s}ov\'{a} and M. Stein, Antipaths in oriented graphs, Discrete Math., 346, 2023, Article 113515.

\bibitem{kst} P. K{\H{o}}v{\'a}ri, V.T. S\'os and P. Tur\'an, On a problem of Zarankiewicz, Colloq. Math., 33, 1954, 50--57. 


\bibitem{Moon-Moser} J.W. Moon and L. Moser, On cliques in graphs, Israel J. Math., 3, 1965, 23--28. 


\bibitem{T} A. Taylor, The regularity method for graphs and digraphs, arXiv:1406.6531.

\bibitem{vala} P. Valadkhan, Extremal oriented graphs and Erd\H os-Hajnal conjecture, Master’s thesis, Simon Fraser University, 2009.

\end{thebibliography}
\end{document}